\documentclass{amsart} 

\RequirePackage[abspath,realmainfile]{currfile}

\usepackage{amsmath,amssymb,amsthm,amsfonts,
amsbsy}
\usepackage{graphicx}
 
\usepackage{enumerate}
\usepackage[final]
{showkeys}

\usepackage{datetime} 
\usepackage{xcolor}

\usepackage{mathrsfs}

\usepackage{hyperref}

\newtheorem{theorem}{Theorem}[section]

\newtheorem{corollary}[theorem]{Corollary}

\theoremstyle{definition}

\theoremstyle{remark}

\numberwithin{equation}{section}

\providecommand{\url}[1]{#1}

\renewcommand{\le}{\leqslant}
\renewcommand{\ge}{\geqslant}

\newcommand{\intr}[2]{\overline{#1,#2}}

\renewcommand{\P}{\operatorname{\mathsf{P}}} 
\newcommand{\E}{\operatorname{\mathsf{E}}}

\newcommand{\ii}[1]{\operatorname{\mathsf{I}}\{#1\}}

\newcommand{\R}{\mathbb{R}}

\renewcommand{\SS}{\mathbb{S}}

\newcommand{\Ga}{\Gamma}

\newcommand{\vp}{\varepsilon}

\newcommand{\ka}{\kappa}

\newcommand{\tX}{\tilde X}

\newcommand{\om}{\omega}
\newcommand{\Om}{\Omega}

\newcommand{\PP}{\mathcal{P}}

\newcommand{\fl}[1]{\lfloor#1\rfloor}

\newcommand{\nnn}[1]{{\left\vert\kern-0.25ex\left\vert\kern-0.25ex\left\vert #1 
    \right\vert\kern-0.25ex\right\vert\kern-0.25ex\right\vert}}

\begin{document}

\title{Multidimensional probability inequalities via spherical symmetry}


\author{Iosif~Pinelis}
\address{Michigan Technological University, USA}
\email{ipinelis@mtu.edu}
%

\subjclass[2010]{60E15; 60B11; 60E10; 60E05}

\keywords{Spherical symmetry; probability identities and inequalities;  von~Bahr--Esseen-type inequalities; random vectors in Hilbert spaces; $p$th~moments of the norm}

%

\begin{abstract}
Spherical symmetry arguments are used to produce a general device to convert identities and inequalities for the $p$th absolute moments of real-valued random variables into the corresponding identities and inequalities for the $p$th moments of the norms of random vectors in Hilbert spaces. Particular results include the following: (i) an expression of the $p$th moment of the norm of such a random vector $X$ in terms of the characteristic functional of $X$; (ii) an extension of a previously obtained von~Bahr--Esseen-type inequality for real-valued random variables with the best possible constant factor to random vectors in Hilbert spaces, still with the best possible constant factor; (iii) an extension of a previously obtained inequality between measures of ``contrast between populations'' and ``spread within
populations'' to random vectors in Hilbert spaces.
\end{abstract}

\maketitle

\section{Introduction}\label{intro}
In this note, we will use averaging with respect to spherically symmetric measures over $\R^d$ to extend certain probability identities and inequalities for real-valued random variables (r.v.'s) to random vectors in Hilbert spaces. 

Let $\mu$ be a finite spherically symmetric measure over $\R^d$, so that $\mu(TB)=\mu(B)$ for all orthogonal transformations $T\colon\R^d\to\R^d$ and Borel sets $B\subseteq\R^d$. Let $g\colon\R\to\R$ be a Borel-measurable function such that $\int_{\R^d}\mu(dt)|g(t\cdot x)|<\infty$ for all $x\in\R^d$, where 
$t\cdot x$ denotes the dot product of vectors $t$ and $x$ in $\R^d$. Let $e_1:=(1,0,\dots,0)\in\R^d$. 
Let $|\cdot|$ denote the Euclidean norm on $\R^d$, for all natural $d$. 

Then, for any $x\in\R^d\setminus\{0\}$ and the unit vector $y:=x/|x|$, 
\begin{equation*}
\begin{aligned}
	H_{g,\mu}(x)&:=\int_{\R^d}\mu(dt)g(t\cdot x) \\ 
	&=\int_{\R^d}\mu(dt)g(|x|t\cdot y) \\ 
	&=\int_{\R^d}\mu(dt)g(|x|t\cdot e_1)=:h_{g,\mu}(|x|), 
\end{aligned}
\end{equation*}
and also $H_{g,\mu}(0)=g(0)\int_{\R^d}\mu(dt)=h_{g,\mu}(0)$,
so that the equality 
\begin{equation}\label{eq:H,h}
	H_{g,\mu}(x)=h_{g,\mu}(|x|) 
\end{equation}
holds for all $x\in\R^d$. It follows that the function $H_{g,\mu}$ is spherically symmetric. 

In particular, for any real $p>0$ and the function $g_p\colon\R\to\R$ defined by the formula 
\begin{equation*}
	g_p(u):=|u|^p
\end{equation*}
for real $u$, identity \eqref{eq:H,h} can be rewritten as 
\begin{equation*}
	|x|^p=\frac1{c_{p,\mu}}\,\int_{\R^d}\mu(dt)|t\cdot x|^p
\end{equation*}
for $x\in\R^d$,  
assuming that 
\begin{equation*}
	c_{p,\mu}:=\int_{\R^d}\mu(dt)|t\cdot e_1|^p\in(0,\infty);
\end{equation*}
the latter condition will hold if e.g.\ $\mu$ is the uniform distribution on the unit sphere $\SS_{d-1}$ in $\R^d$. 

It follows immediately that, for any random vector $X$ in $\R^d$, 
\begin{equation}\label{eq:p,mu}
	\E|X|^p=\frac1{c_{p,\mu}}\,\int_{\R^d}\mu(dt)\E|t\cdot X|^p.  
\end{equation}
Thus, the $p$th moment of the norm $|X|$ of the random vector $X$ is expressed as a mixture of the $p$th absolute moments of the real valued r.v.'s $t\cdot X$. 

Another kind of integral representation of $\E|X|^p$, for any random vector $X$ in $\R^d$ with $\E|X|^p<\infty$, in terms of the distributions of the real-valued r.v.'s $t\cdot X$ for $t\in\R^d$ will be established in Theorem~\ref{th:cf} in Section~\ref{via c.f.}.  

\section{Moments of the norm of a random vector in $\R^d$ via the characteristic functional}\label{via c.f.}

Let $X$ be a random vector in $\R^d$, with the characteristic functional (c.f.) $f_X$, so that 
\begin{equation*}
	f_X(t)=\E e^{it\cdot X}
\end{equation*}
for $t\in\R^d$. 

Take any real $p>0$ which is not an even integer, and let 
\begin{equation*}
	m:=\fl{p/2}. 
\end{equation*}

For nonnegative integers $k$ and real $z$, let 
\begin{equation}\label{eq:c,s}
	c_k(z):=\cos z-\sum_{j=0}^k(-1)^j\frac{z^{2j}}{(2j)!},\quad  
	s_k(z):=\sin z-\sum_{j=0}^k(-1)^j\frac{z^{2j+1}}{(2j+1)!}. 
\end{equation}
Noting that $c'_k=-s_{k-1}$ and $s'_k=c_{k-1}$ if $k\ge1$, and $c_k(0)=0=s_k(0)$, it is easy to check by induction that $(-1)^k c_k(z)\le0$ and $(-1)^k z s_k(z)\le0$, again for all nonnegative integers $k$ and real $z$. 

So, for any nonzero vector $x\in\R^d$ and the unit vector $y:=x/|x|$, by the Tonelli theorem we have 
\begin{align*}
\int_{\R^d}\frac{dt}{|t|^{p+d}}\,c_m(t\cdot x)
&=\int_{\R^d}\frac{dt}{|t|^{p+d}}\,c_m(|x|t\cdot y)	\\ 
&=|x|^p\int_{\R^d}\frac{ds}{|s|^{p+d}}\,c_m(s\cdot y)	\\ 
&=|x|^p\,\Om_d\int_{\SS_{d-1}}d\om\int_0^\infty \frac{dr\,r^{d-1}}{r^{p+d}}\,c_m(r\om\cdot y),  
\end{align*}
where 
\begin{equation}
	\Om_d:=\frac{2\pi^{(d+1)/2}}{\Ga((d+1)/2)} 
\end{equation}
is the surface area of the unit sphere $\SS_{d-1}$ in $\R^d$ and 
$\int_{\SS_{d-1}}d\om\,\cdots$ is the integral with respect to the uniform distribution on $\SS_{d-1}$. 
Using now the spherical symmetry of the uniform distribution on $\SS_{d-1}$ and 
recalling the definition $e_1:=(1,0,\dots,0)\in\SS_{d-1}$, we see that for any nonzero vector $x\in\R^d$ 
\begin{align*}
\int_{\R^d}\frac{dt}{|t|^{p+d}}\,c_m(t\cdot x) 
&=C_{d,p}|x|^p,  
\end{align*}
where 
\begin{align}
	C_{d,p}&:=\Om_d\int_{\SS_{d-1}}d\om\int_0^\infty \frac{dr}{r^{p+1}}\,c_m(r\om\cdot e_1) \notag \\ 
	&=\Om_d\int_{\SS_{d-1}}d\om\int_0^\infty \frac{dr}{r^{p+1}}\,c_m(r|\om\cdot e_1|) 
\label{eq:c even}	\\ 
	&=\Om_d K_p\int_{\SS_{d-1}}d\om\,|\om\cdot e_1|^p  \notag 
\end{align}
and
\begin{equation}\label{eq:K}
	K_p:=\int_0^\infty \frac{dz}{z^{p+1}}\,c_m(z)=-\frac\pi{2\Ga(p+1)\sin(\pi p/2)}; 
\end{equation}
the equality in \eqref{eq:c even} holds because the function $c_m$ is even, and 
the latter equality in \eqref{eq:K} is the special case of formula~(5) in \cite{bahr-converg65} obtained by replacing $x$ there by $1$. 


To evaluate the integral $\int_{\SS_{d-1}}d\om\,|\om\cdot e_1|^p$, note that the distribution of $\om\cdot e_1$ coincides with that of 
\begin{equation*}
	\frac{G_1}{(G_1^2+\dots+G_d^2)^{1/2}}, 
\end{equation*}
where $G_1,\dots,G_d$ are independent standard normal random variables. Hence, \break 
$(\om\cdot e_1)^2$ has the beta distribution with parameters $1/2,(d-1)/2$. So, 
\begin{align*}
	\int_{\SS_{d-1}}d\om\,|\om\cdot e_1|^p
&	=\frac{\Ga(d/2)}{\Ga(1/2)\Ga((d-1)/2)}\,\int_0^1 dz\, z^{p/2-1/2}(1-z)^{(d-1)/2-1} \\ 
&=\frac{\Ga(d/2)\Ga((p+1)/2)}{\sqrt\pi\,\Ga((p+d)/2)}. 
\end{align*}

It follows that 
\begin{equation}\label{eq:C_d,p}
	C_{d,p}=-\frac{\pi^{d/2}}{\sin(\pi p/2)}
	\frac{\Ga(d/2)\Ga((p+1)/2)}{\Ga((d+1)/2)\Ga(p+1)\Ga((p+d)/2)}
\end{equation}
and 
\begin{equation}\label{eq:|x|^p}
	|x|^p=\frac1{C_{d,p}}\,\int_{\R^d}\frac{dt}{|t|^{p+d}}\,c_m(t\cdot x) 
\end{equation}
for all nonzero vectors $x\in\R^d$. Identity \eqref{eq:|x|^p} also trivially holds if $x$ is the zero vector, because $c_m(0)=0$. 

Thus, using the Tonelli theorem again, we immediately obtain the following expression of the $p$th moment of the norm $|X|$ of a random vector $X$ in terms of the c.f.\ $f_X$ of $X$: 

\begin{theorem}\label{th:cf}
For any real $p>0$ which is not an even integer and any random vector $X$ in $\R^d$ with $\E|X|^p<\infty$, 
\begin{equation*}
\begin{aligned}
	\E|X|^p&=\frac1{C_{d,p}}\,\int_{\R^d}\frac{dt}{|t|^{p+d}}\,
	\Big(\Re f_X(t)-\sum_{j=0}^m\frac{(-1)^j}{(2j)!}\,\E(t\cdot X)^{2j}\Big) \\ 
	&=\frac1{C_{d,p}}\,\int_{\R^d}\frac{dt}{|t|^{p+d}}\,
	\Big(\Re f_X(t)-\sum_{j=0}^m\frac{(D_t^{(2j)}f_X)(0)}{(2j)!}\Big), \\ 
\end{aligned}
\end{equation*}
where $C_{d,p}$ is as in \eqref{eq:C_d,p} and  $(D_t^{(2j)}f_X)(0):=\frac{d^{2j}f_X(zt)}{dz^{2j}}\Big|_{z=0}$, the $(2j)$th ``directional'' derivative of $f_X$ at $0$ along the vector $t$. 
\end{theorem}

Formula~(6) in \cite{bahr-converg65} is the special, ``one-dimensional'' case of Theorem~\ref{th:cf} corresponding to $d=1$. 
(For other expressions of the absolute and positive-part moments of r.v.'s in terms of the corresponding characteristic functions, see \cite{positive,pos-part-c.f.s_publ} and references there.) 


Another special case of Theorem~\ref{th:cf} is given by 

\begin{corollary}\label{cor:0<p<2}
For any $p\in(0,2)$ and any random vector $X$ in $\R^d$, 
\begin{equation}\label{eq:|x|^p,p<2}
\begin{aligned}
	\E|X|^p&=\frac1{C_{d,p}}\,\int_{\R^d}\frac{dt}{|t|^{p+d}}\,
	\big(\Re f_X(t)-1\big). 
\end{aligned}
\end{equation}
(If $\E|X|^p=\infty$, then the right-hand side of \eqref{eq:|x|^p,p<2} is $\infty$ as well.)
\end{corollary}

\section{An exact von~Bahr--Esseen-type inequality in Hilbert spaces}\label{bahr-esseen}

Given any sequence $(S_j)_{j=1}^n$ of random vectors in a separable Hilbert space $H$ with the corresponding norm $\|\cdot\|$ on $H$, let $X_j:=S_j-S_{j-1}$ denote the corresponding differences, for $j\in\intr1n$, with the convention $S_0:=0$, so that $X_1=S_1$;  
here and in what follows, for any $m$ and $n$ in the set $\{0,1,\dots,\infty\}$ we let $\intr mn$ stand for the set of all integers $i$ such that $m\le i\le n$. 

If $\E\|X_j\|<\infty$ and $\E(X_j|S_{j-1})=0$ for all $j\in\intr2n$, let us say that the sequence $(S_j)_{j=1}^n$ is a \emph{v-martingale} (where ``v'' stands for ``virtual''); in such a case, let us also say that $(X_1,\dots,X_n)$ is a \emph{v-martingale difference sequence}, or simply that the $X_j$'s are v-martingale differences. 
Note that, for a general v-martingale difference sequence $(X_1,\dots,X_n)$, $X_1$ may be any random vector in $H$; in particular, the expectation $\E X_1$ of $X_1$ (if it exists) may or may not be $0$. 
It is clear that any martingale $(S_j)_{j=1}^n$ is a v-martingale.  

In the rest of this section, it will be always assumed that $(S_j)_{j=1}^n$ is a \break 
v-martingale in $H$. 

Take any $p\in(1,2]$. 

In the case when $H=\R$, the best possible constant $C_p$ in the von~Bahr--Esseen-type inequality 
\begin{equation}\label{eq:AFA}
	\E|S_n|^p\le\E|X_1|^p+C_p\sum_{j=2}^n\E|X_j|^p  
\end{equation}
was obtained in \cite{bahr-esseen-AFA_publ} -- see formula~(1.11) there. 
Specifically, it was shown in \cite[Proposition 1.8]{bahr-esseen-AFA_publ} that 
\begin{equation*}
	C_p=\max_{x\in(0,1)}\ell(p,x), 
\end{equation*}
where 
\begin{equation*}
	\ell(p,x):=(1-x)^p-x^p+px^{p-1}. 
\end{equation*}
It was also shown in \cite{bahr-esseen-AFA_publ} that $C_p$ is continuously and strictly decreasing in $p\in(1,2]$ from $C_{1+}=2$ to $C_2=1$. 

Inequality \eqref{eq:AFA} with $2$ in place of $C_p$ was obtained by von~Bahr and Esseen \cite{bahr65}. The constant factor $2$ was somewhat improved in \cite{bahr65} for $p$ in a left neighborhood of $2$; yet, the improved constant factor was not optimal even in that neighborhood. 
The von~Bahr--Esseen inequality has been extended and used in various kinds of studies; see e.g.\ \cite{bahr-esseen-AFA_publ} for a discussion of some of those results. 

As a corollary of \eqref{eq:AFA}, the following inequality was obtained in \cite[Corollary 1.12]{bahr-esseen-AFA_publ}: 
\begin{equation*}
	\E\nnn{T_n}^p\le(\E\nnn{T_n})^p+\ka_p C_p\sum_{i=1}^n\E\nnn{Y_i-y_i}^p,  
\end{equation*}
where 
\begin{equation*}
	T_n:=\sum_{i=1}^n Y_i,
\end{equation*}
$Y_1,\dots,Y_n$ are any independent random vectors in any separable Banach space $(B,\nnn{\cdot})$, $y_1,\dots,y_n$ are any non-random vectors in $B$, and 
\begin{equation*}
	\ka_p:=\max_{c\in[0,1/2]}\big[(c^{p - 1} + (1 - c)^{p - 1}) \big(c^{\frac1{p - 1}} + (1 - c)^{\frac1{p - 1}}\big)^{p - 1}\big].  
\end{equation*}
It was also shown in \cite{bahr-esseen-AFA_publ} that, somewhat similarly to $C_p$, the factor $\ka_p$ continuously and strictly decreases in $p\in(1,2]$ from $2$ to $1$. 

In this section we will show that inequality \eqref{eq:AFA}, previously obtained in the case $H=\R$, extends to arbitrary separable Hilbert spaces $H$, with the same best possible constant $C_p$: 

\begin{theorem}\label{th:bahr-esseen}
\begin{equation}\label{eq:H}
	\E\|S_n\|^p\le\E\|X_1\|^p+C_p\sum_{j=2}^n\E\|X_j\|^p.  
\end{equation}
\end{theorem}

\begin{proof}
%
Consider first the special case when $H=\R^d$, for some natural $d$. Take any $t\in\R^d$. Then the sequence $(t\cdot S_j)_{j=1}^n$ is a v-martingale in $\R$, because 
\begin{equation}\label{eq:t S}
\begin{aligned}
	\E(t\cdot X_j|t\cdot S_{j-1})=t\cdot\E(X_j|t\cdot S_{j-1})
	&=t\cdot\E((X_j|S_{j-1})|t\cdot S_{j-1}) \\ 
	&=t\cdot\E(0|t\cdot S_{j-1})=0  
\end{aligned}	
\end{equation}
for $j\in\intr2n$. So, by \eqref{eq:AFA}, 
\begin{equation*}
	\E|t\cdot S_n|^p\le\E|t\cdot X_1|^p+C_p\sum_{j=2}^n\E|t\cdot X_j|^p.   
\end{equation*}
Using now the Tonelli theorem and \eqref{eq:p,mu} (with $\mu$ being, say, the uniform distribution on the unit sphere $\SS_{d-1}$ in $\R^d$), we have 
\begin{align*}
	c_{p,\mu}\E\|S_n\|^p&=\E\int_{\R^d}\mu(dt)|t\cdot S_n|^p \\ 
	&=\int_{\R^d}\mu(dt)\E|t\cdot S_n|^p \\ 
	&\le\int_{\R^d}\mu(dt)\Big(\E|t\cdot X_1|^p+C_p\sum_{j=2}^n\E|t\cdot X_j|^p\Big) \\ 
	&=\int_{\R^d}\mu(dt)\E|t\cdot X_1|^p
	+C_p\sum_{j=2}^n\int_{\R^d}\mu(dt)\E|t\cdot X_j|^p \\ 
	&=c_{p,\mu}\Big(\E\|X_1\|^p+C_p\sum_{j=2}^n\E\|X_j\|^p\Big).  
\end{align*}
This proves \eqref{eq:H} for $H=\R^d$ and therefore for any finite-dimensional Euclidean space. 

It remains to consider the case when $H$ is infinite dimensional. 
Let $(e_1,e_2,\dots)$ be any orthonormal basis of the separable Hilbert space $H$. For any natural $d$, let $H_d$ be the span of the set $\{e_1,\dots,e_d\}$, so that $\dim H_d=d<\infty$. Let $P_d$ denote the orthogonal projector of $H$ onto $H_d$. 
Note that the sequence $(P_d S_j)_{j=1}^n$ is a v-martingale in $H_d$ -- cf.\ \eqref{eq:t S}. 
Hence, by what has been proved for finite-dimensional Euclidean spaces, 
\begin{equation}\label{eq:P}
	\E\|P_d S_n\|^p\le\E\|P_d X_1\|^p+C_p\sum_{j=2}^n\E\|P_d X_j\|^p.  
\end{equation}
Note also that, for each $x\in H$, $\|P_d x\|$ is nondecreasing in $d$, and $\|P_d x\|\to\|x\|$ as $d\to\infty$. 
Thus, inequality \eqref{eq:H} follows from \eqref{eq:P} by the monotone convergence theorem. 
\end{proof}

\section{Extensions and applications}\label{extensions}

\subsection{A general device to convert inequalities for the $p$th absolute moments of real-valued random variables into the corresponding inequalities for the $p$th moments of the norms of random vectors in Hilbert spaces}\label{device}

Let $\PP(X_1,\dots,X_n)$ denote some property of random vectors $X_1,\dots,X_n$ in a separable Hilbert space stated in terms applicable to all separable Hilbert spaces. 
Let us say that 
$\PP(X_1,\dots,X_n)$ is \emph{linearly invariant} if for any separable Hilbert space $G$ and any bounded linear operator $A\colon H\to G$ we have the implication 
\begin{equation*}
	\PP(X_1,\dots,X_n)\implies\PP(AX_1,\dots,AX_n).
\end{equation*}

For instance, any one of the following three properties is linearly invariant: 
\begin{enumerate}[(i)]
	\item $(X_1,\dots,X_n)$ is a v-martingale difference sequence;
	\item $X_1,\dots,X_n$ are independent;
	\item $X_1,\dots,X_n$ are symmetric(ally distributed). 
\end{enumerate}

Clearly, the reasoning in 
the proof of Theorem~\ref{th:bahr-esseen} 
can be extended to yield the following general result: 

\begin{theorem}\label{th:general}
Let $\PP(X_1,\dots,X_n)$ be a linearly invariant property of random vectors $X_1,\dots,X_n$ in a separable Hilbert space. 
Suppose that the implication
\begin{equation}\label{eq:implic1}
	\PP(Y_1,\dots,Y_n)\implies\sum_{k=1}^K a_k \E\Big|\sum_{j=1}^n b_{k,j}Y_j\Big|^p\ge0
\end{equation}
holds for some real $p>0$, some natural $K$, some vector $(a_k\colon k\in\intr1K)\in\R^K$, some matrix $(b_{k,j}\colon k\in\intr1K,\,j\in\intr1n)\in\R^{K\times n}$, and all real-valued r.v.'s $Y_1,\dots,Y_n$. Then 
\begin{equation}\label{eq:implic2}
	\PP(X_1,\dots,X_n)\implies\sum_{k=1}^K a_k \E\Big\|\sum_{j=1}^n b_{k,j}X_j\Big\|^p\ge0
\end{equation} 
for all random vectors $X_1,\dots,X_n$ in any separable Hilbert space. 

(For convenience, we assume here that the inequality in \eqref{eq:implic1} always holds if at least one of the $K$ summands in the sum there equals $\infty$ -- even if another summand in that sum equals $-\infty$. The similar convention is assumed for \eqref{eq:implic2}. That is, we assume that the inequalities in \eqref{eq:implic1} and \eqref{eq:implic2} are/can be rewritten in the more general way so that the summands with $a_k>0$ be retained on the left and the summands with $a_k<0$ be moved to the right. This allows us to avoid requiring the finiteness of the $p$th moments.)
\end{theorem}

Thus, Theorem~\ref{th:general} provides a general device to convert inequalities for the absolute $p$th moments of real-valued random variables into corresponding 
inequalities for the $p$th moments of the norms of random vectors in Hilbert spaces. 


Now Theorem~\ref{th:bahr-esseen} can be viewed as a particular case of Theorem~\ref{th:general} -- corresponding to the following settings: $p\in(1,2]$; $\PP(X_1,\dots,X_n)$ meaning that $(X_1,\dots,X_n)$ is a v-martingale difference sequence; $K=n+1$; $a_1=1$, $a_2=\cdots=a_n=C_p$, $a_{n+1}=-1$; $b_{k,j}=\ii{k=j}$ for $(k,j)\in\intr1n\times\intr1n$ and $b_{n+1,j}=1$ for $j\in\intr1n$, where $\ii\cdot$ denotes the indicator. 

\subsection{Applications to von~Bahr--Esseen-type inequalities}\label{Bahr}

Here is another special case of Theorem~\ref{th:general}:  

\begin{corollary}\label{cor:symm}
Take any $p\in[1,2]$. 
Suppose that $(X_1,\dots,X_n)$ is a v-martingale difference sequence in a separable Hilbert space $H$ such that the conditional distribution of $X_j$ given $S_{j-1}$ is symmetric for each $j\in\intr2n$. Then (cf.\ \eqref{eq:H})
\begin{equation*}
	\E\|S_n\|^p\le\sum_{j=1}^n\E\|X_j\|^p.  
\end{equation*} 
\end{corollary}

The particular case of Corollary~\ref{cor:symm} with $H=\R$ is Theorem~1 in \cite{bahr65}. So, Corollary~\ref{cor:symm} follows by Theorem~\ref{th:general}, since the property of $(X_1,\dots,X_n)$ assumed in the second sentence of Corollary~\ref{cor:symm} is linearly invariant. 

Theorem~1 in \cite{bahr65} was obtained as an almost immediate corollary of one of the celebrated Clarkson inequalities -- see the first inequality in \cite[formula~(3)]{clarkson}, reversed for $p\in(1,2]$. That inequality by Clarkson can be stated as 
\begin{equation}\label{eq:clarkson}
\E|x+\vp y|^p\le|x|^p+|y|^p, 	
\end{equation}
where $x$ and $y$ are real numbers and $\vp$ is a Rademacher r.v., so that $\P(\vp=1)=\P(\vp=-1)=1/2$. 

Corollary~\ref{cor:symm} can also be obtained directly, without using \cite[Theorem~1]{bahr65} or Theorem~\ref{th:general} in the present paper. To do that, instead of the Clarkson inequality \eqref{eq:clarkson}, use the inequality 
\begin{equation}\label{eq:ineq symm}
	\E\|x+\vp y\|^p\le\|x\|^p+\|y\|^p
\end{equation}
for $p\in(1,2]$ and $x,y$ in $H$. To prove \eqref{eq:ineq symm}, just note that 
\begin{equation*}
	\E\|x+\vp y\|^p\le(\E\|x+\vp y\|^2)^{p/2}
	=(\|x\|^2+\|y\|^2)^{p/2}
	\le\|x\|^p+\|y\|^p. 
\end{equation*}




\begin{corollary}\label{cor:X-tX}
Take any $p\in[1,2]$. Let $X$ and $\tX$ be independent identically distributed random vectors in $H$. Then 
\begin{equation}\label{eq:symm-ed}
	\E\|X-\tX\|^p\le2\E\|X\|^p.  
\end{equation} 
\end{corollary}

The weaker version of \eqref{eq:symm-ed} with $2^{p-1}$ in place of $2$ follows immediately from the norm inequality $\|X-\tX\|\le\|X\|+\|\tX\|$ and the inequality $(a+b)^p\le2^{p-1}(a^p+b^p)$ for nonnegative $a$ and $b$. 

The special case of \eqref{eq:symm-ed} for $H=\R$ is the first inequality in formula~(10) in Lemma~4 in \cite{bahr65}. Now Corollary~\ref{cor:X-tX} follows by Theorem~\ref{th:general}. 

Alternatively, Corollary~\ref{cor:X-tX} follows from Corollary~\ref{cor:0<p<2}, the same way \cite[Lemma~4]{bahr65} (which is a special case of Corollary~\ref{cor:X-tX}) follows from \cite[Lemma~2]{bahr65} (which is a special case of Corollary~\ref{cor:0<p<2}). 

\medskip

Comparatively recently, the von~Bahr--Esseen inequality was extended to pairwise independent r.v.'s \cite{PChen1,PChen2}. In this regard, let us present 

\begin{theorem}\label{th:pairwise} 
Let $X_1,\dots,X_n$ be pairwise independent zero-mean random vectors in a separable Hilbert space $H$. Then 
\begin{equation}\label{eq:pair}
	\E\|S_n\|^p\le\frac4{2-p}\sum_{j=1}^n\E\|X_j\|^p  
\end{equation}
for $p\in[1,2)$. 

If the $X_j$'s are symmetric, then the factor $\frac4{2-p}$ in \eqref{eq:pair} can be replaced by $\frac2{2-p}$. 
\end{theorem}

\begin{proof}
Consider first the case when $X_1,\dots,X_n$ are symmetric real-valued r.v.'s. Then the r.v.'s $X_j\ii{|X_j|<cx}$ are pairwise independent and zero-mean for all positive $c$ and $x$. Hence, by the standard truncation argument and Chebyshev's inequality, 
\begin{equation*}
	\P(|S_n|\ge x)\le\sum_{j=1}^n\P(|X_j|\ge cx)+\frac1{x^2}\sum_{j=1}^n\E X_j^2\ii{|X_j|<cx}. 
\end{equation*}
So, 
\begin{equation*}
\begin{aligned}
	\E|S_n|^p&=\int_0^\infty px^{p-1}\P(|S_n|\ge x)\,dx \\ 
	&\le\int_0^\infty px^{p-1}\Big(\sum_{j=1}^n\P(|X_j|\ge cx)+\frac1{x^2}\sum_{j=1}^n\E X_j^2\ii{|X_j|<cx}\Big)\,dx \\ 
	&=\frac1{c^p}\sum_{j=1}^n\E|X_j|^p
	+
	\sum_{j=1}^n \E\int_0^\infty px^{p-3}X_j^2\ii{|X_j|<cx}\,dx \\ 
	&=\frac1{c^p}\sum_{j=1}^n\E|X_j|^p
	+
	\sum_{j=1}^n \E X_j^2\int_{|X_j|/c}^\infty px^{p-3}\,dx \\ 
	&=\Big(\frac1{c^p}+\frac p{2-p}\,c^{2-p}\Big)\sum_{j=1}^n\E|X_j|^p. 
\end{aligned}
\end{equation*}
The minimum of the latter expression is attained at $c=1$. 
So, inequality \eqref{eq:pair} with the twice better constant factor $\frac2{2-p}$ (in place of $\frac4{2-p}$) holds when $X_1,\dots,X_n$ are symmetric pairwise independent real-valued r.v.'s. 
So, by Theorem~\ref{th:general}, \eqref{eq:pair} with constant factor $\frac2{2-p}$ holds when $X_1,\dots,X_n$ are symmetric pairwise independent random vectors in $H$, since the property that $X_1,\dots,X_n$ are symmetric pairwise independent random vectors is linearly invariant. 

Finally, by a standard symmetrization argument and in view of Corollary~\ref{cor:X-tX}, \eqref{eq:pair} holds for any pairwise independent zero-mean random vectors $X_1,\dots,X_n$ in $H$. 
%
%
\end{proof}

In the case when $H=\R$, inequality \eqref{eq:pair} was obtained in \cite{PChen1} but with the constant factor 
\begin{equation*}
	K_1(p):=\min_{\vp>0}\Big(\frac{4 (\vp+1)^2}{\vp^2 (2-p)^2}+\vp+2\Big)  
\end{equation*}
in place of $\frac4{2-p}$. 
Note that $K_1(p)$ is a root of a certain polynomial whose coefficients are polynomials in $p$. Clearly, $K_1(p)\ge\frac{4}{(2-p)^2}+2$, so that $K_1(p)$ explodes to $\infty$ inversely-quadratically fast as $p\uparrow2$. 
The constant factor $K_1(p)$ was partially improved \cite[Corollary~2.1]{PChen2} to 
\begin{equation*}
	K_2(p):=\frac{4}{p-1}+\frac{8}{2-p}.  
\end{equation*}
Indeed, $K_2(p)$ explodes to $\infty$ only inversely-linearly fast as $p\uparrow2$. However, in contrast with $K_1(p)$, $K_2(p)$ explodes to $\infty$ also when $p\downarrow1$. 

It is not hard to see that, for all $p\in[1,2)$, the constant factor $\frac4{2-p}$ in \eqref{eq:pair} is at least as twice as good (that, at least as twice as small) as the combined factor $K_{12}(p):=\min(K_1(p),K_2(p))$. This remark is illustrated in Figure~\ref{fig:}. 
%
\begin{figure}[htbp]
	\centering
		\includegraphics[width=.6\textwidth]{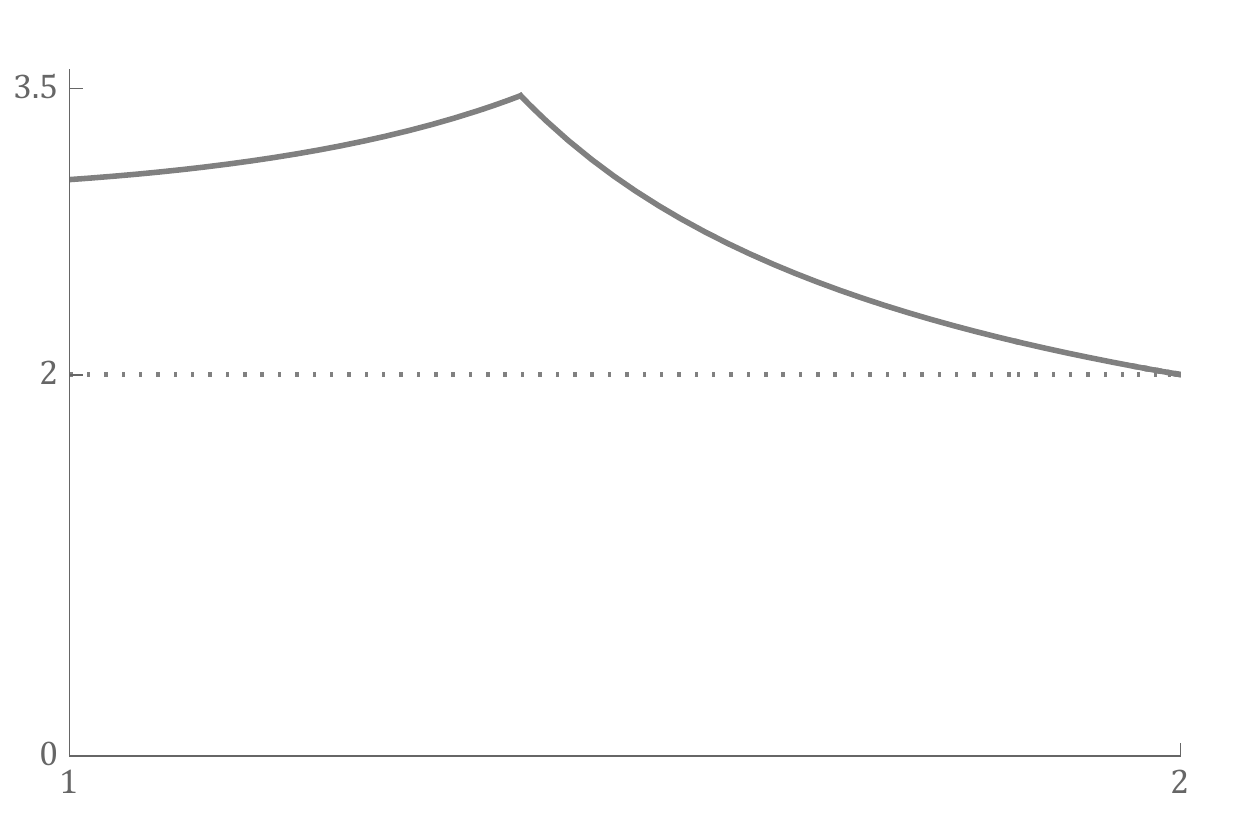
		}
	\caption{Graph of the ratio $K_{12}(p)\Big/\dfrac4{2-p}$ for $p\in(1,2)$.} \label{fig:}
\end{figure}

One may also recall here that the just mentioned results in \cite{PChen1,PChen2} were obtained only for $H=\R$. 

It remains 
an \emph{open question}
whether it is possible to get rid of an explosion to $\infty$ for $p\uparrow2$ of the constant factor in the von~Bahr--Esseen-type inequality for pairwise independent zero-mean r.v.'s (or, equivalently, for pairwise independent zero-mean random vectors in any separable Hilbert space). Ideas from \cite{exactly-one-hardcopy} -- where it was shown that the condition of pairwise independence may be starkly different from the complete independence condition -- 
may turn out to be of use in this regard. However, whatever constant factor one would be able to obtain in a von~Bahr--Esseen-type inequality for pairwise independent random vectors $X_j$ in  $H=\R$, the same constant factor would be available for any separable Hilbert space $H$.  

Pairwise independence plays a notable role in theoretical computer science; see e.g.\ \cite{luby-wigderson}. 

\subsection{Contrast between populations versus spread within populations for vector measurements}

Yet another corollary of Theorem~\ref{th:general} is as follows: 

\begin{corollary}\label{cor:spread}
Take any real $p\in(0,4]$. 
Let $X$ and $Y$ be independent random vectors in a separable Hilbert space $H$. Let $X_1,X_2$ be independent copies of $X$, and let $Y_1,Y_2$ be independent copies of $Y$. 
Assume that $\E\|X\|^p+\E\|Y\|^p<\infty$. For $p>2$, assume also that $\E X=\E Y$. 

Let 
\begin{equation*}
	D_p:=\E\|X-Y\|^p-\tfrac12(\E\|X_1-X_2\|^p+\E\|Y_1-Y_2\|^p). 
\end{equation*}
Then 
\begin{equation*}
	D_p\ge0\text{ for }p\in(0,2],\quad D_p\le0\text{ for }p\in[2,4]. 
\end{equation*}
\end{corollary}

Here $\|X_1-X_2\|^p$ measures the spread of a vector measurement within an $X$ population, and similarly $\|Y_1-Y_2\|^p$ measures the spread of a vector measurement within a $Y$ population, so that $\tfrac12(\E\|X_1-X_2\|^p+\E\|Y_1-Y_2\|^p)$ is the average spread within the two populations. On the other hand, $\|X-Y\|^p$ measures the contrast between vector measurements in the $X$ population and in the $Y$ population. We see that the comparison between the ``average spread within'' and the ``contrast between'' depends on whether $p$ is less or greater than $2$. 

The particular case of Corollary~\ref{cor:spread} with $H=\R$ is Theorem~1 in \cite{contrast-spread-published}. So, Corollary~\ref{cor:spread} follows by Theorem~\ref{th:general}, since the properties of $(X,Y,X_1,X_2,Y_1,Y_2)$ assumed in the first paragraph of Corollary~\ref{cor:spread} are linearly invariant. 

Corollary~\ref{cor:spread} can also be deduced from Theorem~\ref{th:cf}. Indeed, repeat the proof of Theorem~1 in \cite{contrast-spread-published} almost literally -- except for using (i)  Theorem~\ref{th:cf} of the present paper instead of formula (4) in \cite{positive} and (ii) the characteristic functionals instead of the corresponding characteristic functions. Thus, we will prove Corollary~\ref{cor:spread} for $H=\R^d$ and therefore for any finite-dimensional Euclidean space. Finally, to complete the proof of Corollary~\ref{cor:spread}, reason as in the last paragraph of Section~\ref{bahr-esseen}. 

The special case of Corollary~\ref{cor:spread} for $p=1$ and $H=\R^d$ is part of Corollary~2 in \cite{contrast-spread-published}.

%
%


%

%


\bibliographystyle{amsplain}

\bibliography{C:/Users/ipinelis/Documents/pCloudSync/mtu_pCloud_02-02-17/bib_files/citations04-02-21}

\end{document}